\DeclareSymbolFont{AMSb}{U}{msb}{m}{n}
\documentclass[12pt,a4paper,twoside,leqno,noamsfonts]{amsart}
\usepackage[utopia]{mathdesign}
\usepackage[utf8]{inputenc}
\usepackage[english]{babel}
\usepackage{amsmath, amsthm, amstext}
\usepackage{microtype}

\usepackage{indentfirst}
\usepackage[colorlinks,bookmarks]{hyperref}
\usepackage{braket}
\usepackage{caption}
\usepackage{stmaryrd}
\usepackage{comment}
\usepackage{multirow}
\usepackage{booktabs}
\usepackage[all]{xy}
\usepackage{graphicx}

\usepackage{listings}

\newtheorem{cor}{Corollary}[section]
\newtheorem*{cor*}{Corollary}
\newtheorem{lem}[cor]{Lemma}
\newtheorem*{lem*}{Lemma}
\newtheorem{thm}[cor]{Theorem}
\newtheorem*{thm*}{Theorem}
\newtheorem{conj}[cor]{Conjecture}
\newtheorem*{conj*}{Conjecture}
\newtheorem{prop}[cor]{Proposition}
\newtheorem*{prop*}{Proposition}

\theoremstyle{definition}
\newtheorem{defn}[cor]{Definition}

\theoremstyle{remark}
\newtheorem{rmk}[cor]{Remark}
\newtheorem{exa}[cor]{Example}

\newcommand{\bG}{\mathbb{G}}

\newcommand{\bP}{\mathbb{P}}

\newcommand{\cD}{\mathcal{D}}

\newcommand{\cF}{\mathcal{F}}
\newcommand{\cG}{\mathcal{G}}

\newcommand{\cV}{\mathcal{V}}
\newcommand{\cX}{\mathcal{X}}
\newcommand{\cY}{\mathcal{Y}}
\newcommand{\cW}{\mathcal{W}}

\newcommand{\sC}{\mathscr{C}}

\newcommand{\sF}{\mathscr{F}}

\newcommand{\sX}{\mathscr{X}}
\newcommand{\sY}{\mathscr{Y}}

\renewcommand{\dim}{\operatorname{dim}}

\newcommand{\ord}{\operatorname{ord}}

\newcommand{\sm}{{\operatorname{sm}}}
\newcommand{\sing}{{\operatorname{sing}}}

\newcommand{\aut}{\operatorname{Aut}}

\title{Monodromy of projections of hypersurfaces}

\author{Maria Gioia Cifani}
\address[M.G.C.]{Department of Mathematics 'F. Casorati', University of Pavia, via Ferrata 5, 27100 Pavia, Italy}
\email{mariagioia.cifani01@universitadipavia.it}

\author{Alice Cuzzucoli}
\address[A.C.]{Department of Mathematics, University of Warwick, Coventry, CV4 7AL, Warwickshire, England}
\email{a.cuzzucoli@warwick.ac.uk}

\author{Riccardo Moschetti}
\address[R.M]{Department of Mathematics 'F. Casorati', University of Pavia, via Ferrata 5, 27100 Pavia, Italy}
\email{riccardo.moschetti@unipv.it}

\begin{document}
\begin{abstract}
Let $X$ be an irreducible, reduced complex projective hypersurface of degree $d$. A point $P$ not contained in $X$ is called uniform if the monodromy group of the projection of $X$ from $P$ is isomorphic to the symmetric group $S_d$. We prove that the locus of non--uniform points is finite when $X$ is smooth or a general projection of a smooth variety. In general, it is contained in a finite union of linear spaces of codimension at least $2$, except possibly for a special class of hypersurfaces with singular locus linear in codimension $1$. Moreover, we generalise a result of Fukasawa and Takahashi on the finiteness of Galois points.
\end{abstract}
\maketitle

\section{Introduction}
This paper aims at studying the monodromy group of projections of irreducible and reduced complex projective hypersurfaces.
Several accounts of this study, in particular in the case of projective curves, can be found in literature. 
The classical \textit{Uniform Position Principle} due to Castelnuovo, in the formulation of Harris \cite{JHCurves}, can be applied to show that the monodromy group of a general projection of a curve is the symmetric group. It has been proved in \cite{GN, GM, GS} that the monodromy group of an indecomposable projection of a general curve of genus greater than $3$ is either the symmetric or the alternating group. More recently, all the monodromy groups of projections of a smooth planar curve of degree smaller or equal than $5$ have been classified (see \cite{MY1, Miura1, Y}).\\
This paper is motivated by the work of Pirola and Schlesinger \cite{PS}, in which the authors consider the monodromy group of projections of any irreducible reduced projective curve. In our case, we fix a hypersurface $X\subset\bP^{n+1}$ of degree $d$ and we consider all its natural linear projections $\pi_P$ from a point $P\notin X$. We classify such source points by means of the monodromy group of the associated projection. Our final goal is a characterisation of the locus of points for which the associated monodromy group is strictly contained in the symmetric group $S_d$. 

\begin{defn}
The point $P$ is called \emph{uniform} if $M(\pi_P)\cong S_d$ and \emph{non--uniform} otherwise.
We denote by $\cW(X)$ the locus of non--uniform points of $X$.
\end{defn}

Specifically, we are interested in two classes of non--uniform points: a point $P \in \cW(X)$ is called \emph{Galois} if the field extension associated with the map $\pi_P$ is a Galois extension. A point $P \in \cW(X)$ is called \emph{decomposable} if the projection $\pi_P$ is decomposable, i.e. if it factors via two morphisms of degree greater than $1$. The loci of Galois and decomposable points are denoted by $\cG(X)$ and $\cD(X)$, respectively. 

The non--uniform locus $\cW(X)$ is constructible and a first step in order to understand its structure is to compute its dimension. The Uniform Position Principle extends to hypersurfaces by taking general hyperplane sections: indeed, for any irreducible, reduced hypersurface $X \subset \bP^{n+1}$ we have $\dim \cW(X) < n+1$. This bound is far from optimal. The case of curves has been solved by Pirola and Schlesinger: in the work \cite{PS}, the authors proved  that the locus of non--uniform points associated with projections of an irreducible reduced plane curve is finite. The case of surfaces in $\bP^3$ has been partially covered in \cite{CMS}, where it is proved that the locus of non--uniform points of a smooth surface in $\bP^3$ is finite. All these bounds are optimal.

The main result of this paper describes a property of the locus of non--uniform points for projective hypersurfaces:

\begin{thm}\label{thm:main}
Let $X$ be an irreducible, reduced hypersurface of $\ \bP^{n+1}$, $n \geq 2$. Then $\cW(X)$ is contained in a finite union of linear spaces of codimension $2$ in $\bP^{n+1}$, unless $X^\sing$ is the union of at least $2$ components isomorphic to $\bP^{n-1}$. In this case, $\cW(X)$ must be a union of rational curves, lying in the intersection of the tangent cones to points in $X^\sing$.
\end{thm}
Cones over planar curves which admit non--uniform points give examples of hypersurfaces $X$ for which $\cW(X)$ has codimension $2$ (see Example \ref{ex:codim2}). The possibility of $\cW(X)$ being a union of rational curves seems unlikely, and its discussed in Remark \ref{rem:types}. As a consequence of the main result, we give in Proposition \ref{prop:smooth} a bound on the dimension of the locus of non-uniform points for smooth varieties; in particular, such locus is finite for smooth hypersurfaces.

\begin{thm}\label{thm:smooth}
Let $X$ be a smooth, complex projective hypersurface of dimension $n$ in $\bP^{n+1}$. Then the locus of non--uniform points is finite.
\end{thm}

The primary tool in the proofs is the theory of focal loci of families of lines in $\bP^{n+1}$, a classical topic dating back to Segre (\cite{Segre}). For our specific problem, we prove a generalisation of \cite[Proposition 4.3]{CF} and of \cite[Lemma 2]{DePoi}, described in Lemma \ref{lem:focalnewtuttoinsieme} and Lemma \ref{lemma:depoi}, respectively. 

We also use the focal machinery to give a property of the monodromy group when $\cW(X)$ is infinite and $X$ is not a cone. 

\begin{thm} \label{thm:transpositionnotcone}
If $\dim \cW(X)>0$ and $X$ is not a cone, then the monodromy group associated with all but finitely many points of $\cW(X)$ contains transpositions.
\end{thm}

As a consequence of this result, we can give a characterisation of the two loci $\cG(X)$ and $\cD(X)$. Galois points have been introduced and extensively studied in various works, for instance \cite{Yoshiara, FT, FukCom, FukAut} to name a few. These works were particularly focused on computing the number of Galois points. Yoshihara gives many examples of smooth hypersurfaces $X$ with non-empty $\cG(X)$ in \cite[Proposition 11]{Yoshiara} (see also \cite{FT} for more general examples in case of normal hypersurfaces). We give a generalisation of a result of Fukasawa and Takahashi (\cite[Proposition 6]{FT}).

\begin{prop} \label{prop:analogogalois}
Let $X$ be an irreducible, reduced hypersurface in $\bP^{n+1}$ of degree $d \geq 3$. Then $\cG(X)$ is finite unless $X$ is a cone.
\end{prop}

Decomposable maps were studied in many different contexts as in \cite{BastianelliCortiniDePoi, BCFS}. A consequence of Theorem \ref{thm:transpositionnotcone} is that the case $\cW(X)$ of positive dimension when $X$ is not a cone depends only on the locus $\cD(X)$.

\begin{prop}\label{Prop:decomposablenoncone}
Let $X$ be an irreducible, reduced hypersurface in $\bP^{n+1}$ which is not a cone. Then $\cW(X) \smallsetminus \cD(X)$ is finite. 
\end{prop}

Notice that for all hypersurfaces $X$ having prime degree, the locus $\cD(X)$ is empty. These results serve as evidence of the following:

\begin{conj} \label{conj:coni}
Let $X$ be an irreducible, reduced hypersurface of $\ \bP^{n+1}$, $n\geq 1$. Then the locus $\cW(X)$ is finite unless $X$ is a cone.
\end{conj}

\medskip

\textbf{Plan of the paper.} In Section \ref{sec:preliminaries}, we recall some basic definitions which will be useful in the following. The theory of focal loci, which is the main technical tool used in this paper, is introduced in Section \ref{Fuochi}. The main results in this context are Lemma \ref{lem:focalnewtuttoinsieme} and Lemma \ref{lemma:depoi}. The first part of Section \ref{section:nonuniform} is devoted to the proof of Theorem \ref{thm:main} about non--uniform points, followed by the consequences concerning general projections of smooth varieties. The second part contains the study of families of simply tangent lines passing through $\cW(X)$ (see Theorem \ref{thm:transpositionnotcone}) and their relationship with $\cG(X)$ and $\cD(X)$. We conclude the paper by summarising our results in light of Conjecture \ref{conj:coni}, this is done in Remark \ref{rem:types}.
\medskip

\textbf{Notation.}
From now on, the varieties are assumed to be complex and projective. Let $\sF$ be a family of objects parametrised by a scheme $V$. We say the general element of $\sF$ satisfies a certain property $P$ if $P$ holds for every element in a Zariski dense open subset of $V$. 
We will use the notation $\bG(r, \bP^n)$ for the Grassmannian parametrising linear spaces of projective dimension $r$ contained in $\bP^n$. We denote by $X^\sing$ the singular locus of a variety $X$ and by $X^\sm$ its complement.

\section{Preliminaries} \label{sec:preliminaries}

\subsection{Projections}
In this section we will briefly recall some results useful for various proofs in Section \ref{section:nonuniform}.

\subsection{Monodromy}\label{subsec:monodromy}
We can define the monodromy group of a finite dominant morphism $f:X \to Y$ of degree $d>1$ between complex irreducible reduced hypersurfaces in $\bP^{n+1}$ as follows. 
Let $U \subset Y$ be a Zariski open set over which 
$f$ is \'etale, and let $y$ denote a point in $U$.
We have a well defined map
$$\mu: \pi_1(U,y) \to \aut\big(f^{-1}(y)\big).$$
The image $M(f):=\mu\left(\pi_1(U,y)\right)$ is called \emph{monodromy group} of the map $f$; it is a transitive subgroup of $\aut(f^{-1}(y))\simeq S_d$.

We can also describe this group by means of Galois extensions: let $K$ be the Galois closure of the extension $k(X)/k(Y)$, where $k(X),k(Y)$ define the fields of rational functions of $X$ and $Y$, respectively. Define the \emph{Galois group} $G(f)$ of the map $f$ to be the Galois group of the field extension $K/k(Y)$.
It turns out that $G(f)$ is isomorphic to $M(f)$, see \cite[Section I]{H}. It follows that $M(f)$ does not depend on the choices of $U$ and $y$.

\subsection{Projections}
Let $X \subset \bP^{n+c}$ be a smooth variety of dimension $n$. Let $T$ be a linear subspace of dimension $t\leq c-1$ such that $T \cap X= \emptyset$. Consider the finite map defined by the natural linear projection $\pi_T: X \to \bP^n$. The following theorem is the algebraic version of a result of Mather in \cite{Mather}, which provides a powerful tool in order to understand singularities arising from projections. We will apply it in Proposition \ref{prop:genproj} and \ref{prop:smooth}, for $t=c-2$ and $t=c-1$, respectively.

\begin{thm}\cite[Theorem 1]{abo} \label{thm:Mather}
Let $X$ and $T$ as above. For any $i_1 \leq t+1$, define $X_{i_1}:=\{x \in X\ |\ \dim(T_xX \cap T)=i_1-1\}$. When $X_{i_1}$ is smooth, define $X_{i_1,i_2}:=\{x \in X_{i_1}\ |\ \dim(T_xX_{i_1} \cap T)=i_2-1\}$ and so on. When possible, define $X_{i_1,\ldots,i_k}$ for $i_k \leq \ldots\leq i_2 \leq i_1$. For $T$ general, every $X_{i_1,\ldots,i_k}$ is smooth and, when not empty, its codimension is positive, see \cite[Theorem 2]{ao} for details.
\end{thm}

We recall here the definition of projective cone and the statement of Bertini's Theorem. 

\begin{defn}\label{defn:cone}\cite[Ex 3.1]{harrisAG}
Let $\Gamma \cong \bP^k $ be a linear subspace in $\bP^{n+1}$ and let $Y \subset \Gamma$ be a variety. Let $V\cong \bP^{n-k}$ be a linear subspace disjoint from $\Gamma$. The \emph{cone} over $Y$ of vertex $V$ is a variety in $\bP^{n+1}$ defined as the union of the lines joining the vertex $V$ with points of $Y$. 
\end{defn}

The notion of cone will be used in Lemma \ref{lemma:depoi} and in Theorem \ref{thm:transpositionnotcone} in conjunction with the notion of dual variety. An introduction to this topic can be found in \cite{ProjDual}.

\begin{thm}\cite[Theorem 3.3.1]{Laz1}\label{Bertini1}
Let $X$ be an irreducible variety and $f: X \to \bP^r$ a morphism. Fix an integer $d < \dim \overline{f(X)}$. If $L \subset \bP^r$ is a general $(r-d)$-plane, then $f^{-1}(L)$ is irreducible.
\end{thm}

When dealing with a projective hypersurface $X$, the map $f$ as defined in Theorem \ref{Bertini1} above is the inclusion $X \hookrightarrow \bP^{n+1}$.  
Moreover, we can fix a point $Q$ in $\bP^{n+1} \smallsetminus X$ and chose a general $(n+1-e)$-plane $L$ (where $e < \dim X$) passing through $Q$. Then the result still holds for $X \cap L$ by following the same lines of the proof and applying the following.

\begin{lem}\label{Bertini2}\cite[Lemma 3.3.2]{Laz1}
Let $p:Y \to S$ be a dominant morphism between irreducible complex varieties. Assume that $p$ admits a section $s:S \to Y$ whose image does not lie in the singular locus of $Y$. Then the fibre $Y_b:=p^{-1}(b)$ is irreducible for any general point $b \in S$.
\end{lem}

For what concerns the monodromy group, we have the following version of Bertini-type theorem.

\begin{lem}\label{lem:section}
Let $X$ be an irreducible variety in $\bP^{n+1}$ of dimension $n\geq 2$. Let $H$ be a general linear subspace of codimension $k \geq 1$, and $X_H$ the section of $X$ cut by $H$. Then, for a point $P\in H$ such that $P \notin X$, we have
\begin{equation*} 
\xymatrix{
X_H \ar[d]_{{\pi_P}_{|X_H}}  \ar@{^{(}->}[r]^i & X \ar[d]^{\pi_P} & \\
\bP^k  \ar@{^{(}->}[r]^i &  \bP^n.
}
\end{equation*}
As a consequence we have 
$$M({\pi_P}_{|X_H}) \leq M(\pi_P).$$
\end{lem}
 
\subsection{Families of tangent lines} \label{sec:ftl}
This section is devoted to the study of families of lines that are tangent to a hypersurface. 
We will briefly go through some preliminaries to highlight the aspects related to our problem; for a more general introduction see for instance \cite{Zak}.  

Consider a reduced, irreducible hypersurface $X \subset \bP^{n+1}$ of degree $d$ and a line $l \nsubseteq X$. The intersection $X \cap l$ consists of a finite number of points $P_1, \ldots, P_k$ counted with multiplicities $m_1, \ldots, m_k$ such that $\sum m_i=d$.\\
We recall some terminology that is useful to characterise the family of lines with respect to the hypersurface.
\begin{defn}
We call the \emph{contact order} of $l$ with $X$ at $P_i$ the number $m_i-1$, and we denote it by $\ord_{P_i}(l \cap X)$. The line $l$ is \emph{transverse} to $X$ at $P_i$ if $\ord_{P_i}(l \cap X) = 0$, and \emph{tangent} to $X$ at $P_i$ if $\ord_{P_i}(l \cap X) \geq 1$. In the case of higher contact order, i.e. $\ord_{P_i}(l \cap X) \geq 2$, we say that the line $l$ is \emph{asymptotic tangent} to $X$ at $P_i$. The line $l$ is called \emph{bitangent} to $X$ at two points $P_i \neq P_j$, if $l$ is tangent to $X$ at both points $P_i,P_j$. We say that $l$ is a \emph{simple tangent} if there is a unique tangent point $P_i\in l\cap X$ with $\ord_{P_i}(l \cap X) = 1$ and $l$ is transverse to $X$ for all the other $P_j\neq P_i$ in $l\cap X$.
\end{defn}

Notice that if we take a singular point in $X$, then all the lines passing through it will be at least simply tangent.

\begin{defn}\cite[Lecture 20]{harrisAG} \label{defn:vartanglines}
Consider an hypersurface $X$ and a point $P \in X$. Choose an affine neighborhood of $P$ where $P$ is the origin. In this neighborhood, $X$ is described by a certain polynomial $f:=f_m+f_{m+1}+ \cdots$, where $f_k$ is homogeneous of degree $k$, and $m$ is the smallest integer such that $f_m$ is not vanishing. The \emph{tangent cone} to $X$ at the point $P$ is the hypersurface described by the polynomial $f_m$.
\end{defn}

\subsection{Branch locus}
Let $X$ be a reduced and irreducible hypersurface in $\bP^{n+1}$ of degree $d$, and $\pi_P:X \to \bP^n$ be the projection from a point $P\notin X$ and let $y$ be a point in $\bP^n$. The fibre over $y$ is defined by the set $\pi_P^{-1}(y)=\{P_1,\ldots,P_k\}$ consisting of $k \leq d$ distinct points. This set corresponds to the set-theoretical intersection of $X$ with the line through $P$ and $y$.
\begin{defn}
We call classical branch locus of $\pi_P$ the locus $\mathcal{B}$ of points $y \in \bP^n$ such that the cardinality of the fibre $\pi_P^{-1}(y)$ is strictly lower than $d$.
\end{defn}

The image via $\pi_P$ of $X^\sing$ is contained in the classical branch locus, since any line passing through $X^\sing$ is tangent to $X$. 
We want to relate the branch locus of $\pi_P$ to the lines in $\bP^{n+1}$ for which the tangency order is greater than the order of the general line in $\bP^{n+1}$.
To this end, consider the normalisation map $\nu:\tilde{X} \to X$, and its composition $\phi$ with the projection $\pi_P$:
\begin{equation*}
    \phi=\pi_P\circ \nu:\tilde{X}\rightarrow \bP^n.
\end{equation*}
Let $y$ be a general point of an irreducible component of $\mathcal{B}$ and let $l$ be the $0$ dimensional subscheme of $X$ obtained by cutting $X$ with the line $\langle P, y \rangle$. Thus, its pullback to $\tilde{X}$ is given by $\tilde{l}= m_1x_1+\ldots+m_tx_t$, where $d \geq t \geq k$ and $\sum_{i=1}^{t}m_i=d$. 
The image of the singular locus of $\tilde{X}$ via $\phi$ is in codimension $2$ in $\bP^n$. As $y$ is chosen to be general, the points $x_1,\ldots,x_t$ are smooth in $\tilde{X}$.

\begin{defn} \label{def:branch}
An irreducible component of $\mathcal{B}$ is called a \emph{branch component} if the fibre on $\tilde{X}$ of a general point $y$ has at least a (necessarily smooth) point $x_i$ with $m_i \geq 2$. The union of all the branch components is called the \emph{branch locus} of $\pi_P$. We will denote it by $B_P$.
Moreover, we can define the \emph{branching weight} of a point $y$ in the branch locus as
$$b(y):=\sum_{i=1}^t (m_i -1) \geq 1.$$
We say that $y \in B_P$ is a \emph{simple branch point} if $b(y)=1$. 
\end{defn}

There is a relationship between the branching weight and the permutation type of the corresponding element in $S_d$ via the monodromy map. In particular, simple branch points correspond to transpositions in the monodromy group, see \cite[Section II.3]{H}. The following lemma is proved in \cite[Lemma 4.6]{Miranda}. We report its statement in the particular case of projections by using the notation introduced before.

\begin{lem}  \label{lem:permutations}
Consider the projection $\pi_P$ of a irreducible reduced planar curve $X$ from a point $P \notin X$. Let $y$ be a point of $B_P$. The cycle structure of the permutation representing a small loop around $y$ in the monodromy group is $(m_{1}, \ldots, m_{t})$, where the $m_i$ are defined as above.
\end{lem}

\section{Focal loci of a family of lines}\label{Fuochi}
In this section we discuss some special families of lines related to our hypersurface which will be important for their relationship with the monodromy group. We refer the reader to \cite[Chapter III.9]{Hartshorne} and to \cite[Chapter 4.6.7]{Sernesi} for background material about families of algebraic spaces. \\

Let $\cX$ be a flat family of closed subschemes of $\bP^{n+1}$ parametrised by a integral base scheme $S$. This can be described by the following diagram, where the map $i$ is the inclusion and $p,q$ are the projections on the first and second factor respectively:
\begin{equation*} 
\xymatrix{
\cX \ar[d]_{p|_{\cX}} \ar[dr]_{f} \ar@{^{(}->}[r]^-i & S \times \bP^{n+1} \ar[d]^{q} & \\
S &  \bP^{n+1}.
}
\end{equation*}
\begin{defn}
The kernel of the differential $\mathrm{d}f$ defines a sheaf $\cF$ over $\cX$ and it is called \emph{focal sheaf}. The locus $\cF(\cX)$, i.e. the support of the sheaf, is called the \emph{focal scheme} or, more classically, \emph{focal locus}.
\end{defn}

From now on we will consider only the case in which $\cX$ is a family of lines in $\bP^{n+1}$. We can think of $S$ as a subscheme of the Grassmannian $\bG(1,\bP^{n+1})$. The family $\cX$ can be interpreted via the map $f$ as the subset of the points in $\bP^{n+1}$ of the corresponding lines. 

\begin{defn} \label{defn:fillingfamily}
The family $\cX$ of lines in $\bP^{n+1}$ over the base $S$ is called \emph{filling family} if the dimension of $S$ is $n$ and the map $f=i \circ q$ is dominant.
\end{defn}
The focal locus $\cF(l_s)$ restricted to a general line $l_s$ of the family $\cX$ is where the following map has not maximal rank 
\begin{equation*}
    \mathcal{O}_{l_s}^{\oplus n}\cong T_{S,s}\otimes \mathcal{O}_{l_s} \to \mathcal{N}_{l_s | \bP^{n+1}} \cong \mathcal{O}_{l_s}(1)^{\oplus n}
\end{equation*}

Let $\cX$ be a filling family of lines in $\bP^{n+1}$, so that $\dim \cX=n$.
Assume $\cX$ is locally parametrised by $S:=S(u_1, \ldots, u_{n})$. 
The line $l_s$ corresponding to a point $s \in S$ can be described by the intersection of $n$ distinct hyperplanes
\begin{equation} \label{eqn:linea}
l_s:=\{a_1(s) \cdot \underline{x}= \ldots = a_{n}(s) \cdot \underline{x}=0\}.
\end{equation}
Here $\underline{x}=(x_0:\ldots:x_{n+1})$ is the vector of the coordinates in $\bP^{n+1}$ and $a_i(s)=(a_i(s)_0:\ldots :a_i(s)_{n+1})$ determines the i-th hyperplane. 
We will denote by $\partial_{u_k} a_i(s)_j$ the partial derivative of $a_i(s)_j$ with respect to the variable $u_k$, and inductively for high order derivatives $\partial_{u_k, u_l} a_i(s)_j$, and so on. In the following we will omit the dependency on $s$, by writing for instance just $\partial_{u_i} a_i$ for the vector $(\partial_{u_i} a_1(s)_0 :\ldots : \partial_{u_i} a_i(s)_{n+1})$. 
With this notation, the equation of the focal locus on the line $l_s$ is
\begin{equation} \label{Eqn:schemafocale}
\det \begin{pmatrix}
(\partial_{u_1} a_1) \cdot \underline{x} & \cdots & (\partial_{u_{n}} a_{1}) \cdot \underline{x}\\
\vdots & & \vdots\\
(\partial_{u_1} a_{n}) \cdot \underline{x} & \cdots & (\partial_{u_{n}} a_{n}) \cdot \underline{x}\\
\end{pmatrix} = 0,
\end{equation}
together with the equations in (\ref{eqn:linea}).

\begin{lem} \label{lem: degree focal scheme} \cite[Proposition 4.3]{CF}
Let $\cX$ be a filling family of lines in $\bP^{n+1}$ and let $s \in S$ be a general point of the base. Then the focal locus in the fibre $l_s$ consists of $n$ points counted with the right multiplicity as root of Equation (\ref{Eqn:schemafocale}).
\end{lem}

\begin{defn}
A point $P$ in $\bP^{n+1}$ is called \emph{fundamental} for the family $\cX$ if there is a subfamily $\sX'$ of lines all passing through it. The \emph{fundamental locus} is the subset of $\bP^{n+1}$ of fundamental points. 
\end{defn}

The following facts on fundamental points are well known, and their origins date back to Segre, in the work \cite{Segre}:
 
\begin{lem}\label{lem:fundpoint}
Consider a filling family of lines in $\bP^{n+1}$ and assume there is a subfamily $\sX'$ of lines all passing through a point $P$. If the dimension of the base of $\sX'$ is $k$, then $P$ is a focus of multiplicity $k$.
\end{lem}

The number of lines of a family $\cX$ through a general point of $\bP^{n+1}$ is classically called the \emph{order} of the family $\cX$. 

\begin{prop}\cite[Proposition 2.7]{DePoi}\label{prop:depoi1}
Let $\cX$ be a filling family of lines in $\bP^{n+1}$ of order $1$. Then, the focal locus coincides with the fundamental locus. 
\end{prop}

We now prove two results on the fundamental locus of particular filling families of lines in $\bP^{n+1}$ that we will use later. 
The following result is a generalisation of parts (a) and (c) of Proposition 5.1 in \cite{CF}. 

\begin{lem} \label{lem:focalnewtuttoinsieme}
Consider a filling family $\cX$ of lines in $\bP^{n+1}$ and an irreducible reduced hypersurface $X$. Assume the general $l \in \cX$ tangent to $X$ at a general point $P$. 
Then $P$ is a focus on $l$. Moreover, if the contact order of $l$ with $X$ at $P$ is at least $2$, then $P$ is a focus with multiplicity at least $2$ on $l$.
\end{lem}
\begin{proof}
We can assume that the hypersurface $X$ is parametrised locally around $P$ by the same $S$ which parametrises the family $\cX$, so $P:=P(u_1, \ldots, u_n)$. Moreover we can choose $a_1$ to define the tangent plane to $X$ at $P$. 
So we have
$$a_1 \cdot P = a_1 \cdot (\partial_{u_1} P) = \ldots = a_1 \cdot (\partial_{u_{n}} P)=0.$$
By taking partial derivatives and by using the previous relations, we get
\begin{equation} \label{Eqn:relationstangent}
(\partial_{u_1} a_1) \cdot  P = \ldots = (\partial_{u_{n}} a_1) \cdot  P =0.
\end{equation}
It immediately follows that $P$ satisfies Equation (\ref{Eqn:schemafocale}), and so is a focus on $l$ independently of the choices of the other $n-1$ hyperplanes.

Notice that the line $l$ is tangent to $X$, so we can assume $l$ to be explicitly parametrised as follows:
\begin{equation} \label{Eqn:explicitformofl}
l:= \left\{P + \lambda \partial_{u_1} P\right\}.
\end{equation}

Assume now that the contact order of $l$ with $X$ at $P$ is at least $2$, hence we get
$$a_1 \cdot (\partial_{u_1, u_1} P) = 0.$$
For any hyperplane $\{b \cdot \underline{x}=0\}$ passing through $P$ and containing $l$, we have that $b \cdot (\partial_{u_1} P)=0$ from Equation (\ref{Eqn:explicitformofl}), and hence, by taking derivatives, we get that 
\begin{equation} \label{Eqn:otherplanes}
(\partial_{u_1} b) \cdot  P=0.
\end{equation}

We can choose other $n-1$ independent vectors $b_2, \ldots, b_{n}$ in $\bP^{n+1}$ and the hyperplanes ${b_i \cdot \underline{x}=0}$, defined by passing through $P$ and containing $l$.
Hence, such line is given by the equations:
\begin{equation}\label{eqn:rettatangente}
    \left\{a_1 \cdot \underline{x} = b_2 \cdot \underline{x}=\ldots =b_{n} \cdot \underline{x}=0\right\},
\end{equation}
and the focal scheme on $l$ is described by Equation (\ref{Eqn:schemafocale}).

Now we can use Equation (\ref{Eqn:explicitformofl}) to express the focal scheme as a function of the parameter $\lambda$. Moreover, if we consider Equation (\ref{Eqn:relationstangent}) and Equation (\ref{Eqn:otherplanes}), we get a simplified form for our matrix:

\begin{equation} \label{eqn:matricefinale}
\det \begin{pmatrix}
0                                                     & \lambda (\partial_{u_2} a_1) \cdot (\partial_{u_1} P) & \cdots & \lambda (\partial_{u_{n}} a_1) \cdot (\partial_{u_1} P)\\
\lambda (\partial_{u_1} b_2) \cdot (\partial_{u_1} P)     & \cdots                                            & \cdots & \cdots\\
\vdots                                                &                                                   &        & \vdots\\
\lambda (\partial_{u_1} b_{n}) \cdot (\partial_{u_1} P) & \cdots                                            & \cdots & \cdots\\
\end{pmatrix} = 0. 
\end{equation}

Lemma \ref{lem: degree focal scheme} guarantees that the determinant is not identically zero whenever $l$ is a general element of the family. Such a determinant is given by $\lambda^2 \cdot \alpha (\lambda)= 0$, where $\alpha$ is a polynomial depending on $\lambda$. Hence the point $P$ is a focus of multiplicity at least $2$.
\end{proof}

In the proof of Theorem \ref{thm:main} we will have to deal also with a family of lines obtained by joining a curve and a variety of codimension $2$. At this purpose we generalise \cite[Lemma 2]{DePoi} to higher dimension.

\begin{lem} \label{lemma:depoi}
Let $F$ be a codimension $2$ subvariety of $\bP^{n+1}$, and $\sC\nsubseteq F$ be a curve not contained in a $\bP^{n-1}$. Assume that the family $\cX$ of lines joining $\sC$ and $F$ is filling. Then $F$ is linear and $\sC$ is rational. If $\sC\cap F= \emptyset$, $\sC$ is also linear. Otherwise, $F$ meets $\sC$ in $\deg(\sC) -1$ points. 
\end{lem}
\begin{proof}
In this proof we will follow \cite{DePoi}, in particular we will denote the cone of vertex $V$ over $Y$ by $\chi_{Y,V}$. We also refer to it for the notations and preliminaries about Schubert cycles.
The points of $X^\sing$ are fundamental, hence by Proposition \ref{prop:depoi1} and Lemma \ref{lem: degree focal scheme}, the focal locus has codimension exactly $2$. We can apply \cite[Theorem 2.1]{Depoi2} and conclude that the order of $\cX$ must be $1$. Let $d_1$ and $d_2$ be the degree of $\sC$ and $F$, respectively.
Assume first that $\sC \cap F \neq \emptyset$, and let $m$ be the number of lines through a general point $P$ passing through $\sC \cap F$. 
The cone $\chi_{\sC,P}$ has dimension $2$ and degree $d_1$, while the cone $\chi_{F,P}$ has dimension $n$ and degree $d_2$. They intersect in $d_1d_2$ lines through the point $P$, and we know that one of these lines belongs to the family $\cX$. Notice that this line does not pass through $\sC \cap F$ by construction. Therefore, we have that $m = d_1d_2 - 1$.
Now consider the cone $\chi_{\sC,Q}$, where $Q$ is a general point of $\sC$, which has degree $d_1 -1$. Recall that the general line of $\cX$ meets $\sC$ (resp. $F$) at a single point and those are the only focal points on the line. The cone $\chi_{\sC,Q}$ is not contained in the cone $\chi_{Q,F}$ of degree $d_2$, since otherwise the lines secant to $\sC$ will intersect $F$, and hence will all be composed of focal points.
As before, the intersection of the two cones gives $d_2(d_1-1)$ lines through $Q$ meeting $\sC \cap F$: if one of these lines meets $\sC$ and $F$ in distinct points, then all the line would be focal. Therefore, we have that $m=d_2(d_1-1)$. 
Summing up, $d_2d_1-1 = d_2d_1-d_2$ which gives that $d_2=\deg(F)=1$.

We are left with the case  $\sC \cap F = \emptyset$. By following the notation of \cite{DePoi}, let $M(F)$ be the family of lines in $\bP^{n+1}$ meeting $F$. It is a codimension $1$ family in $\bG(1, \bP^{n+1})$, hence it can be written as $$M(F)=d_2 \sigma_{1,0}.$$
In the same way, let $M(\sC)$ be the family of lines meeting $\sC$. It is a codimension $n-1$ family in $\bG(1, \bP^{n+1})$, hence it can be written as 
$$M(\sC)=\sum_{t=0}^k a_t \sigma_{n-1-t,t}= a_0\sigma_{n-1,0}+\ldots+a_k\sigma_{n-1-k,k},$$
where $k=\lfloor{\frac{n-1}{2}}\rfloor$. 
The complementary cycle of $\sigma_{n-1-t,t}$ with $t>0$ is $\sigma_{n-t,t+1}$, which gives a family of lines contained in a $\bP^{n-t}$. However, a general $\bP^{n-t}$, $t \geq 1$, does not intersect the curve $\sC$, hence $M(\sC)\cdot \sigma_{n-t,t+1} = a_t=0$ for $t>0$. The complementary cycle of $\sigma_{n-1,0}$ is $\sigma_{n,1}$, which gives a family of lines contained in a $\bP^n$ passing through a general point. A general hyperplane cuts $\sC$ in $d_1$ points and the lines passing through them are lines of $M(\sC)$ contained in this hyperplane. Hence $M(\sC)$ reduces to $d_1\sigma_{n-1,0}$. The filling family $\cX$ is given by the intersection of $M(F)$ and $M(\sC)$. Pieri's formula gives 
$$M(F) \cdot M(\sC) =d_1d_2(\sigma_{n,0}+ \sigma_{n-1,1}),$$
saying that through a general point of $\bP^{n+1}$ pass $d_1d_2$ lines of $\cX$. By our assumption, $d_1=d_2=1$, i.e. both $\sC$ and $F$ are linear. The proof of the rationality of $\sC$ follows the same lines of \cite[Theorem 0.7]{DePoi}. 
\end{proof}

\section{The locus of non--uniform points}\label{section:nonuniform}
The first part of this section is devoted to the proof of Theorem \ref{thm:main}. We will apply all the focal machinery developed in the previous section to the following family. 

\begin{defn}
Let $\cV$ be the family in $\bG(1,\bP^{n+1})$ composed by the lines $l$ such that one of the following cases occurs:
\begin{itemize}
    \item[(C1)] The line $l$ is bitangent or asymptotic tangent to $X^\sm$;
    \item[(C2)] The line $l$ passes through a point of $X^\sing$ and is tangent to a point of $X^\sm$;
    \item[(C3)] The line $l$ intersects $X^\sing$ in more than one point;
    \item[(C4)] The line $l$ is in the tangent cone to $X$ at a point in $X^\sing$, see Definition \ref{defn:vartanglines}. 
\end{itemize}
If $\sY$ is a variety in $\bP^{n+1}$ we define $\cV_\cY$ the subfamily of $\cV$ of lines through $\cY$.
\end{defn}

Notice that thanks to Lemma \ref{lem:permutations}, the family of lines which generates the monodromy of non--uniform points is contained in $\cV_{\cW(X)}$.

\begin{lem}\label{lem:dimGq}
Let $X$ be a irreducible, reduced hypersurface in $\bP^{n+1}$, and let $Q\in \cW(X)$ be a non--uniform point. Then, the base parametrising the family $\cV_Q$ has dimension $n-1$ in the Grassmannian $\bG(1,n+1)$.
\end{lem}
\begin{proof}
For $n=1$, \cite[Proposition 2.5]{PS} guarantees that a non--uniform point must have at least two non--simple tangent lines passing through it. We proceed now by induction. Assume that the claim is true for a hypersurface of $\dim X=n-1$ and prove it for the case $\dim X=n$. By contradiction, assume that the dimension of the base of $\cV_Q$ is smaller than $n-1$. Take a general hyperplane $H$ in $\bP^{n+1}$ passing through $Q$; by Bertini's Theorem, the section $X \cap H$ is irreducible and reduced since $Q \notin X$. The hyperplane $H$ meets the family $\cV_Q$ in a subfamily parametrised by a base of dimension strictly lower than $n-1$, but this contradicts the induction hypothesis.
\end{proof}

\begin{lem} \label{lem:fillingG}
Let $\sC \subset \cW(X)$ be an irreducible curve  not contained in a linear space of codimension $2$. Then the family $\cV_\sC$ is filling.
\end{lem}
\begin{proof}
We want to show that $\cV_\sC$ is a family whose base space has dimension $n$ and the map $\cV_\sC \to \bP^{n+1}$ of Definition \ref{defn:fillingfamily} is dominant. For every choice of $Q \in \sC$, the dimension of the base of $\cV_Q$ is $n-1$ thanks to Lemma \ref{lem:dimGq}. Every line in $\cV_\sC$ belongs to the cone $\cV_Q$ for a certain $Q \in \sC$, so
the dimension of the base of $\cV_\sC$ is $n$. 

If the map $\cV_\sC \to \bP^{n+1}$ of Definition \ref{defn:fillingfamily} was not dominant, then the union of all the $\cV_Q$ would be contained a finite union of divisors in $\bP^{n+1}$. For a general $Q \in \sC$, $\cV_Q$ is the union of cones over $V_j \cap X$ with vertex $Q$. Let us consider $V_j:=V$ for a $j \in \{1,\ldots,r\}$. The cone $\cV_Q = \cV_Q'$ for every $Q, Q' \in \sC$ and we will just write $V \cong \cV_Q$ for every $Q \in \sC$. We claim that $V$ is linear. Consider a general line $l$ passing through a general point $T \in \sC$ and not contained in $V$. If $V$ were not a hyperplane, there should be at least a point $Z \in V \cap l,\ Z \neq T$. By hypothesis, $V$ is the cone over $V \cap X$ with vertex $T$. The line $\langle Z, T \rangle$ with $Z \in V$ is contained in $V$. This is a contradiction. Hence $V$ is linear.

As a consequence, the curve $\sC$ must be contained in the intersection of $H_1, \ldots, H_k$, so we have to rule out the case in which $k=1$. A general $\bP^2$ passing through a general point $Q$ of $\sC$ will intersect $X$ in an irreducible, reduced curve. 
The hyperplanes $H_i$ intersect this last $\bP^2$ in lines which, together with some transposition coming from lines outside the $H_i$'s, are generators of the monodromy group $\pi_Q$. Recall that the product of all the generators is the identity: if $k=1$, the monodromy group $\pi_Q$ will be generated just by the transpositions outside the $H_i$'s.  However this can not happen, as the point $Q$ is not uniform. Consequently, there must be at least two generators coming from the $H_i$, hence $k > 1$.
\end{proof}

The problem of finding a bound on the dimension of $\cW(X)$ for planar curves has been completely solved in \cite{PS}:

\begin{thm}\cite[Theorem 3.5, $r=2$]{PS}\label{thm:basecase}
Let $X \subset \bP^2$ be an irreducible curve. Then the locus of non--uniform points is finite.
\end{thm}

Notice that by subsequently taking general hyperplane sections (see Lemma \ref{lem:section}), Theorem \ref{thm:basecase} implies that the codimension of $\cW(X)$ must be at least $2$. We want to prove that this locus of codimension at least $2$ must also be contained in a finite union of linear spaces. We are now ready for the proof of the main result of this paper.

\begin{proof}[Proof of Theorem \ref{thm:main}]
Let us assume that there exists a component of $\cW(X)$ not contained in a linear space of codimension $2$. Consider an irreducible curve $\sC \subset \cW(X)$ with the same property. We now want to apply the focal machinery to the family $\cV_\sC$. The hypothesis of Lemma \ref{lem:fillingG} are satisfied, so we know that $\cV_\sC$ is filling. Let us proceed with a case by case analysis.

\medskip
\textbf{The general element of $\cV_\sC$ belongs to Case (C1)}.
We claim that the general $l$ is tangent to a general point of $X$. Indeed, if it would not be true, there would be a divisor $Y \subset X^{\sm}$ such that the family $\cV_\sC$ of lines tangent to $X$ would be the join of $Y$ and $\sC$. Therefore, $\sC$ would be contained in the tangent hyperplane $T_yX$ for every $y \in Y$. Reasoning as in Lemma \ref{lem:fillingG}, $\sC$ would be contained in a linear space of codimension $2$, that is a contradiction. 

If $l$ is an asymptotic tangent line to a general point $P\in X$, then it is a focal point with multiplicity $2$; if $l$ is bitangent at two distinct points then both of them are focal points for $l$ (see Lemma \ref{lem:focalnewtuttoinsieme}). 

\medskip
\textbf{The general element of $\cV_\sC$ belongs to Case (C2)}.

As for the previous case, we can assume that the point in $l$ is tangent to a general point in $X^\sm$, so the tangency point is a focus by Lemma \ref{lem:focalnewtuttoinsieme}, while a point in $l \cap X^\sing$ is focal by Lemma \ref{lem:fundpoint}.  
     
\medskip   
\textbf{The general element of $\cV_\sC$ belongs to Case (C3)}. The family $\cV_\sC$ consists of lines passing through at least two points of $X^\sing$ and intersecting $\cW(X)$. Hence, there is a one--dimensional subfamily of lines of $\cV_{\sC}$ through every point in $ X^\sing$.
By Lemma \ref{lem:fundpoint}, the points in $l \cap X^\sing$ are focal points for $l$, each of multiplicity $1$.

\medskip
In each of the previous cases, the focal locus of $l$ has multiplicity at least $2$ in points where $l$ is tangent to $X$. Moreover, $l$ passes through a point of $\sC$ and, by construction, for every such point there is a $n-1$ dimensional subfamily of $\cV_\sC$. Therefore, this point is a focus for $l$ as well, its multiplicity being $n-1$ by Lemma \ref{lem:fundpoint}. Note that the general line meets $\sC$ in a point outside $X$. 
Thus, we have at least $n+1$ focal points in a general line $l$ of the filling family $\cV_\sC$. 
But this is a contradiction because of Lemma \ref{lem: degree focal scheme}, as the focal locus in a general line of a filling family of lines in $\bP^{n+1}$ consists of $n$ points counted with multiplicity. 
\medskip

\textbf{The general element of $\cV_\sC$ belongs to Case (C4)}. Having found a contradiction in all the previous cases, we can assume that (C4) is the only case occurring for a general element of $\cV_\sC$.
In this case, a general $l \in \cV_\sC$ belongs to the tangent cone to $X$ at a point $x$ in $X^\sing$. As a consequence, $\sC$ must be contained in the intersection of all the tangent cones to points in $X^\sing$. The family $\cV_\sC$ is compose by lines joining $\sC$ and $X^\sing$. Recall that we are still assuming $\sC$ not contained in a linear space of codimension $2$, hence $\cV_\sC$ is filling.
By applying Lemma \ref{lemma:depoi}, we get that $\sC$ is rational, $X^\sing$ is linear and meets $\sC$ in $\deg(\sC)-1$ points. There could be only a finite number of curves $\sC \subset \cW(X)$ satisfying these properties. Therefore, the dimension of $\cW(X)$ must be at most $1$. The image of $X^\sing$ under the projection from a point is a linear subspace of $\bP^n$ and, by Lemma \ref{lem:permutations} gives generators of the monodromy group of the projection which are not transpositions. Moreover, we are assuming to be only in the case (C4), so $X^\sing$ must split in at least two linear spaces, in order to have a non--uniform monodromy. 
\end{proof}

We remark that the only obstruction to prove that $\cW(X)$ is always contained in a linear space of codimension $2$ comes from the case (C4). Notice that the case of $\cW(X)$ being a finite union of linear space of codimension $2$ can actually happen, as shown in the following example.

\begin{exa}\label{ex:codim2}
Fix an irreducible and reduced curve $C \subset \bP^{n+1}$ contained in a plane $H \cong \bP^2$ and consider a linear space $V$ of dimension $n-2$ disjoint from $H$. Let $X$ be the cone on $C$ with vertex $V$. Assume that $Q$ is a non--uniform point for $C$ in $H$. We claim that every point in the line $\langle Q, V \rangle$ is non--uniform for $X$: just notice that the linear projection from the vertex induces an isomorphism from the general $\bP^2$ to $H$ which sends $X \cap \bP^2$ to $C$. In a similar flavour, cones provide an example of irreducible, reduced $X \subset \bP^{n+1}$ with $\cW(X)$ being a finite union of $\bP^k$, $k=1, \ldots, n-1$. 
\end{exa}

As a consequence of the main theorem, we are able to prove the finiteness of $\cW(X)$ for many hypersurfaces. As a generalisation of Theorem 1.1 in \cite{CMS} we have Theorem \ref{thm:smooth}. We prove it more generally by obtaining the analogous of \cite[Theorem 3.5]{PS} in higher dimension. Consider a smooth, irreducible variety $\tilde{X}$ of dimension $n$ in $\bP^{n+c}$, $c \geq 1$. We can define the monodromy group $M(\pi_L)$ associated with the projection from a linear space $L \in \bG(c-1, \bP^{n+c})$. Our next step is to study the dimension of $\cW(\tilde X)$.
We start with projecting $\tilde X$ to a hypersurface $X \subset \bP^{n+1}$, and to study $\cW(X)$. 

\begin{prop}\label{prop:genproj}
Let $\tilde X$ be a smooth irreducible projective variety of dimension $n$ in $\ \bP^{n+c}$ and $X$ the projection of $\tilde X$ from a general linear subspace $M\subset \bP^{n+c}$ of dimension $c-2$. Then, the locus $\cW(X)$ is at most finite.
\end{prop}
\begin{proof}
Assume $\cW(X)$ is not finite and let $\sC$ be one of its components. Theorem \ref{thm:Mather} implies that the tangent cone to $X$ at a general point in $X^\sing$ is a finite union of hyperplanes.  By following the proof of Theorem \ref{thm:main} we have that $\cW(X)$ is always contained in the intersection of two different tangent cones to $X$. As a consequence, in this case $\sC$ must always be contained in a linear space of codimension $2$. 

Denote by $K\cong \bP^k$ the smallest linear subspace of $\bP^{n+1}$ containing $\sC$ and consider the family of hyperplanes $H\cong\bP^{k+1}$ containing $K$. We claim that the general $H$ in this family cuts $X$ in a reducible hypersurface. If that were not true, $\cW(X \cap H)$ would span a space of codimension $1$ in $H$, which contradicts Theorem \ref{thm:main}. Notice that $X \cap H$ cannot be non--reduced for a general $H$ by Bertini's Theorem \ref{Bertini1}, since we assumed $X$ to be non--reduced. 

Notice that $X \cap H$ is the linear projection from $M$ of $\tilde{X} \cap \langle H, M\rangle$. The variety $\tilde{X} \cap \langle H, M\rangle$ must be reducible as well as the projection is a continuous map. So $\langle H, M\rangle$ gives a family of hyperplanes $\bP^{k+c-1}$ whose general member cuts $\tilde X$ in a reducible variety. The base locus in $\tilde X$ of this family is obtained by intersecting $\tilde X$ with $\Gamma:= \langle K, M\rangle$. We are assuming that the general section on $\tilde X$ obtained from a general element of this linear system is reducible, thus contradicting Lemma \ref{Bertini2}. Therefore, the base locus must consist of singular points. Since $\tilde X$ is smooth, the only possibility is that $\Gamma$ is tangent to $\tilde X$. This contradicts Theorem \ref{thm:Mather}: as we assumed $M$ general, $\Gamma \cap \tilde X$ is contained in one of the smooth varieties $X_{i_1, \ldots, i_k}$ (in the notation of Theorem \ref{thm:Mather}).
\end{proof}

As a consequence of this proposition, we can reason as in \cite[Theorem 3.5]{PS} and show the following

\begin{prop} \label{prop:smooth}
Let $\tilde{X}$ be a smooth irreducible complex projective variety of dimension $n$ in $\bP^{n+c}$, $c \geq 1$. The locus of non--uniform $(c-1)$-planes $L$ not intersecting $\tilde{X}$ has codimension at least $n+1$ in the Grassmannian $\bG(c-1,\bP^{n+c})$. 
\end{prop}
\begin{proof}
We will follow the proof of \cite[Theorem 3.5]{PS}. When $c=1$ we know from Proposition \ref{prop:genproj} that all but finitely many points $P\in \bP^{n+1} \smallsetminus \tilde{X}$ are uniform.
Now assume $c \geq 2$. After projecting from a general $(c-2)$-subspace $M$, we get $X \subset \bP^{n+1}_M$, where $\bP^{n+1}_M$ parametrises all the ($c-1$)-planes $L$ containing $M$. Notice that projecting $\tilde{X}$ to $\bP^n$ from $L$ is equivalent to projecting $X$ to $\bP^n$ from the point in $\bP^{n+1}_M$ corresponding to $L$. Proposition \ref{prop:genproj} applied to $X$ gives $\dim \cW(X) \leq 0$.

Assume by contradiction that $\cW(\tilde X)$ has codimension at most $n$ in the Grassmannian $\bG(c-1,\bP^{n+c})$. In this case there would be an irreducible subvariety $D$ of codimension at most $n$ such that the general $L \in D$ is non--uniform, i.e. every $L \in D \smallsetminus Z$ is non--uniform for a proper Zariski closed subset $Z$. 
We claim that for a general element $M \in \bG(c-2,\bP^{n+c})$, the dimension of $(D \smallsetminus Z) \cap \bP^{n+1}_M$ is greater than zero. Notice first that $D \cap \bP^{n+1}_M$ is at least one-dimensional: $D$ has codimension at most $n$ in $\bG(c-1,\bP^{n+c})$ and $\bP^{n+1}_M$ is $n+1$ dimensional. 
Secondly, we have that
$$ \dim (Z \cap \bP^{n+1}_M)< \dim(D \cap \bP^{n+1}_M).$$ 
This implies that there exist infinitely many non--uniform such planes $L$ containing a general $M\in \bG(c-2,\bP^{n+c})$, but this would give $\dim(X)>0$, which contradicts Proposition \ref{prop:genproj}. 
\end{proof}

\begin{rmk} \label{rmk:boundsharp}
The same argument of \cite[Remark 3.6]{PS} shows that the bound in Proposition \ref{prop:smooth} is sharp. There are singular 
varieties $X$ for which there exist points $x\notin X$ such that the projection $\pi_x:X \subset \bP^{n+c} \to \bP^{n+c-1}$ is non-birational onto the image. If this is the case, a $(c-1)$-plane $L$ containing such an $x$ is non--uniform because the map $\pi_L$ factorises in a non-trivial way. Thus, the family of the $(c-1)$-planes passing through $x$ is a family consisting of non--uniform elements in $\bG(c-1,\bP^{n+c})$ of codimension $n+1$.
\end{rmk}

We now focus on the case in which $\cW(X)$ is not finite. Transpositions play a fundamental role in determining if a point is uniform or not, see \cite[Remark 2.2]{PS}. This motivates Theorem \ref{thm:transpositionnotcone}, in which as in the case (C1) of \ref{thm:main}, we apply Lemma \ref{lem:focalnewtuttoinsieme} to show that if $\dim \cW(X)>0$ and $X$ is not a cone, then the monodromy group associated with all but finitely many points of $\cW(X)$ contains transpositions.

\begin{proof}[Proof of Theorem \ref{thm:transpositionnotcone}]
Consider the family $\cX$ of lines in $\bP^{n+1}$ tangent to $X$ at smooth points and passing through a curve $\sC$ inside $\cW(X)$.

In order to prove that $\cX$ is filling, notice that $\cX$ is composed by lines lying on hyperplanes tangent to $X$ and passing through $\sC$. Let $X^* \subset (\bP^{n+1})^*$ be the dual variety of $X$; let $r$ be the dimension of $X^*$. If $X$ is not a cone, by \cite[Theorem 1.25]{ProjDual} we have that $X^*$ is not contained in a hyperplane. Consider the family of hyperplanes in $(\bP^{n+1})^*$ dual to the points of $\sC$. 
The general hyperplane of this family intersects $X^*$ in a locus of dimension $r-1$. Moreover, every point of $X^*$ is contained in one of such hyperplanes, providing that the tangent hyperplane to $X$ at a general point pass through the general point of $\sC$.
As a consequence, for the general $Q \in \sC$ we have a $r-1$ dimensional family of tangent hyperplanes to $X$. 
The general member of this family is tangent to $X$ along a subvariety of dimension $n-r$. Every line joining $Q$ and this $n-r$ subvariety is tangent to $X$. 

This is enough to prove that the family $\cX$ is filling: the dimension of the family is the right one. The map $\cX \to \bP^{n+1}$ is dominant because the hyperplanes tangent to $X$ passing trough $\sC$ are tangent at the general point of $X$, so $\cX$ cannot degenerate as in the proof of Lemma \ref{lem:fillingG}. 

If we assume that the general line of $\cX$ is not simply tangent to $X$, we get a contradiction by using Lemma \ref{lem: degree focal scheme}, as in the proof of Theorem \ref{thm:main}. As a consequence, if $X$ is not a cone, we can find simple tangent lines to $X$ passing through all but finitely many points of $\cW(X)$. Such lines correspond to transpositions in the monodromy group. 
\end{proof}

From Theorem \ref{thm:transpositionnotcone} we get  Proposition \ref{prop:analogogalois}, that is a generalization of \cite[Proposition 6]{FT} on Galois points.

\begin{proof}[Proof of Proposition \ref{prop:analogogalois}] 
Following the notation of \cite{FT}, denote by $\Delta'(X)$ the locus of Galois points associated with $X$. Clearly, $\Delta'(X) \subset \cW(X)$. If we assume $\Delta'(X)$ to be an infinite set and $X$ not a cone, Theorem \ref{thm:transpositionnotcone} shows the existence of transpositions in the monodromy group associated with a general point $Q$ in $\Delta'(X)$. As a consequence, the field extension given by $\pi_Q$ is not Galois, and this contradicts our initial hypothesis.
\end{proof}

Recall that if $M(\pi_P)$ is isomorphic to the full symmetric group then the projection $\pi_P$ is indecomposable. The converse also holds if we require $M(\pi_P)$ to contain a transposition (see \cite[Remark 2.2]{PS}). Hence we have Proposition \ref{Prop:decomposablenoncone}.

\begin{proof}[Proof of Proposition \ref{Prop:decomposablenoncone}]
Let $Q$ be a general point in $\cW(X)$. By Theorem \ref{thm:transpositionnotcone} the monodromy group $M(\pi_Q)$ contains a transposition. Hence the projection $\pi_Q$ must be a decomposable map.
\end{proof}

\begin{rmk}
Notice that this is enough to prove Conjecture \ref{conj:coni} 
for all hypersurfaces $X$ having prime degree. Indeed, $\pi_P:X \to \bP^n$ is indecomposable for every $P \notin X$ because otherwise, the degree of an intermediate, not birational map would divide $d$. 
\end{rmk}

There are two classes of hypersurfaces which could potentially provide a counterexample for Conjecture \ref{conj:coni}: \\
\textbf{Type-1:} Hypersurfaces $X$ in $\bP^{n+1}$, where every component of $X^\sing$ in codimension $1$ is linear, and such that the intersection of all the tangent cones at points in $X^\sing$ is a finite union of rational curves.
\\
\textbf{Type-2:} Hypersurfaces $X$ in $\bP^{n+1}$ such that there exists a $\bP^k$ ($0<k<n$), where $X \cap \bP^{k+1}$ is reducible for every $\bP^{k+1} \supset \bP^k$. Cones are a particular case of $X$ of Type-2.

\begin{rmk} \label{rem:types}
We have that if $X$ is neither of Type-1 nor of Type-2, then $\cW(X)$ is finite.
The proof mimics the steps of Proposition \ref{prop:genproj}: if $X$ is not of Type-1 we can apply Theorem \ref{thm:main} to get that that $\cW(X) \subset \bP^{n-1}$; if $X$ is not of Type-2 either, we can apply Lemma \ref{lem:section}.
\end{rmk}

While examples of hypersurfaces of Type-1 and Type-2 do exist, we were not able to find any variety with an infinite number of non--uniform points that is not a cone. We plan to study these cases in a future work.

\section*{Acknowledgements}
R.M. is supported by MIUR: Dipartimenti di Eccellenza Program (2018-2022) - Dept. of Math. Univ. of Pavia.
We would like to thank Gian Pietro Pirola for introducing us to the problem and the different techniques involved; we also thank him for all the help he gave us during the preparation of this paper. We are also grateful to Ciro Ciliberto for constructive discussions about how to generalise some result concerning focal loci to a higher dimension. We benefit from many helpful discussions with Thomas Dedieu, Enrico Schlesinger and Lidia Stoppino.

\bibliographystyle{alpha}
\bibliography{biblio}
 
\end{document}